\documentclass{aims} 
\usepackage{amsmath}
\usepackage{paralist}
\usepackage[misc]{ifsym}
\usepackage{epsfig} 
\usepackage{epstopdf} 
\usepackage[colorlinks=true]{hyperref}
\hypersetup{urlcolor=blue, citecolor=red}
\allowdisplaybreaks

\def\currentvolume{X}
 \def\currentissue{X}
  \def\currentyear{200X}
   \def\currentmonth{XX}
    \def\ppages{X-XX}
     \def\DOI{10.3934/xx.xxxxxxx}

\newtheorem{theorem}{Theorem}[section]
\newtheorem{corollary}[theorem]{Corollary}

\newtheorem{lemma}[theorem]{Lemma}

\theoremstyle{definition}

\newtheorem{remark}[theorem]{Remark}

\title[RAMSEY MODEL OF OPTIMAL GROWTH
WITH ALLEE EFFECT]
{Ramsey model of optimal growth \\  with Allee effect} 

\author[Zihan Wang and Vladimir Shikhman]{}


\thanks{$^*$Corresponding author: Vladimir Shikhman}

\begin{document}
\maketitle

\centerline{\scshape
Zihan Wang $^{{\href{mailto:wang1zihan11@gmail.com}{\textrm{\Letter}}}1}$
and Vladimir Shikhman$^{{\href{mailto:vladimir.shikhman@mathematik.tu-chemnitz.de}{\textrm{\Letter}}}*1}$}

\medskip

{\footnotesize
 \centerline{$^1$ Chemnitz University of Technology, Reichenhainer Str. 41, 09126 Chemnitz, Germany}
} 

\bigskip



\begin{abstract}
For the Ramsey model of economic growth,
which describes the optimal allocation of consumption and saving over
time, we assume the population dynamics to follow the Allee effect. The so-called Allee threshold separates two regimes from each other. If starting below the threshold, the population decreases to zero. Above this threshold, it gradually saturates. We show that the corresponding consumption per capita stabilizes at two different levels accordingly. As for our main result, the consumption per capita performs in the long run better if the population becomes extinct, rather then it advances the saturation level. This is in line with our previous results on the capital per capita for the underlying Solow-Swan model of economic growth with the Allee effect. However, the comparison of consumption-to-capital ratios at the both steady states crucially depends on the curvature of the production function. 
\end{abstract}



\section{Introduction}
The Ramsey model, since its formulation in 1928 \cite{11} and further popularization, has served as a foundational framework in intertemporal macroeconomic theory. It describes the optimal allocation of consumption and saving over time, providing insights into long-run economic growth and capital accumulation. 
The classical model assumes a representative household that chooses consumption to maximize lifetime utility, subject to a capital accumulation constraint and a production function. However, since the underlying population dynamics exhibits constant and positive population growth rate, more complex ecological or demographic factors are neglected.

In this paper, we introduce the Allee effect \cite{15} into the Ramsey framework to investigate how varying and negative population growth rates influence optimal consumption paths and long-run capital dynamics. The Allee effect, originating in ecology, describes a phenomenon in which a population's growth rate can be negative when the population size falls below a critical threshold. This effect can lead to population extinction if the size remains too small. We incorporate this feature by modifying the population dynamics within the Ramsey model.
The main goal of this study is to explore how the Allee effect alters the optimal behavior of agents and the stability properties of the resulting dynamical system. Specifically, we analyze how the presence of an Allee threshold affects capital accumulation, optimal consumption, and the long-run survival of the economy.

Previous studies have examined extensions of the Ramsey model to include demographic dynamics, like logistic growth law \cite{a} or bounded population growth rates \cite{2}, but few have considered critical population thresholds or the implications of ecological feedback. By integrating the Allee effect, we contribute to a growing literature that links ecological constraints with economic decision-making. In particular, in analyzing the long-run behavior of the capital per capita, we need to establish rigorous upper and lower bounds. For this purpose, we make use of the results in \cite{1}, where Akhalaya and Shikhman studied the Solow–Swan model with Allee effect and derived bounding techniques that we adapt in the Ramsey setting. Moreover, our work also relies on the findings in \cite{2}, where Guerrini investigated the Ramsey model under varying, but positive, population growth rate. While his framework did not include the case of negative growth, his contribution on the functional family of the consumption-to-capital ratio provides valuable analytical tools. In particular, Guerrini obtained closed-form solutions under certain conditions, whereas our focus lies on studying the stability of the consumption-to-capital ratio and characterizing its asymptotic behavior. Finally, we mention the paper \cite{b}, where the authors include the Allee effect for biodiversity into a variant of Ramsey model in order to study pollution and mass extinction phenomena.

The structure of the paper is as follows. Section 2 presents the classical Ramsey model and introduces the modified population dynamics incorporating the Allee effect. Then, we analyze the resulting optimal control problem, including the use of Pontryagin’s maximum principle \cite{10}. Section 3 discusses bounding techniques and long-term behavior, while Section 4 concludes with numerical simulations.


\section{Model description}
\label{sec:Model-assumption}
The classical Ramsey model \cite{11} is a fundamental framework in economic growth theory. It describes the intertemporal allocation of consumption and savings in an economy with a representative agent maximizing utility over an infinite horizon. The standard Ramsey model assumes a production function of the form
\begin{equation}
    Y(t) = F(K(t), L(t)),
\end{equation}
where $Y$ is output, $K$ represents the aggregate capital stock, and $L$ represents the labour force with respect to the time $t$. We assume throughout the paper that $F$ is homogeneous of degree one, i.e.
\[
F(a K, a L)=a F(K, L) \quad \text { for all } a>0.
\]
Moreover, we assume $F$ to be strictly concave with respect to $(K,L)$.
As in \cite{6,11},
the economy's output $Y$ is further composed of savings $S$ and of consumption $C \geq 0$. The savings rate will be denoted by $s \geq 0$. This implies that there is a fixed proportion of savings in each unit of output, i.e.
 \begin{equation}
  S = sY = Y - C. 
  \label{S(t)}
 \end{equation}
Note that the savings $S$ go completely for investment $I$, i.e.
 \begin{equation}
     I=S. 
  \label{I(t)}
 \end{equation}
The capital per unit of labour is defined as
\[
    k(t) = \frac{K(t)}{L(t)},
\]
and consumption per unit of labour as
\[
    c(t) = \frac{C(t)}{L(t)}.
\]
The representative agent in the original Ramsey model  chooses consumption per unit of labour $c(t)$ in order to maximize the total welfare function
\begin{equation}
   \label{eq:utility}
    W = \int_0^{\infty} e^{-\rho t} u(c(t)) dt,
\end{equation}
subject to the capital accumulation equation:
\begin{equation}
\label{eq:cap-old}
    \dot{k}(t) = f(k(t)) - c(t) - (\delta + r)k(t),
\end{equation}
where $u$ used in the Ramsey model is the isoelastic utility function:

\begin{equation}
u(c(t)) = \frac{c^{1-1 / \sigma}(t) - 1}{1-1 / \sigma},
\label{CRRA}
\end{equation}
and

\[
f(k) = \frac{F(K, L)}{L} = F\left(\frac{K}{L}, 1\right) = F(k,1)
\]
is the output per unit of labour, 
$\sigma >0$ is constant rate of intertemporal elasticity of substitution,
$\rho >0$ is the discount rate, $\delta >0$ is the depreciation rate, and $r >0$ is the constant population growth rate from
\[
   \dot{L}(t) = r L(t).    
\]

Incorporating the Allee effect into the Ramsey model leads to a modified population dynamics equation. The Allee effect describes a phenomenon in population dynamics, where a population has a critical threshold below which it cannot sustain itself \cite{15}. In our model, we propose the following labour dynamics according to the Allee effect:
\begin{equation}
    \dot{L}(t) = r L(t) \left(1 - \frac{L(t)}{M}\right) \left(\frac{L(t)}{N} - 1 \right),
\end{equation}
where $M$ is the carrying capacity, and $N$ is the critical population threshold. This leads to the modified version of (\ref{eq:cap-old}):
\begin{equation}
\label{eq:cap-new}
    \dot{k}(t) = f(k(t)) - c(t) - (\delta + n(t))k(t),
\end{equation}
with the now variable population growth rate
\[
  n(t) = \frac{\dot{L}(t)}{L(t)}=r\left(1 - \frac{L(t)}{M}\right) \left(\frac{L(t)}{N} - 1 \right).
\]
\begin{remark}[\cite{1}]
Depending on $L(t_0)=L_{0}$, the population growth rate will have different behavior:
\begin{itemize}
    \item[(1)] If $L_{0} \in(0, N)$, then 
    $-r<n(t)<0$, $n^{\prime}(t)<0$ and $\displaystyle\lim_{t \to \infty}n(t)=-r$;
    \item[(2a)] If $L_{0} \in\left(N, \frac{N+M}{2}\right]$, then 
    $
    0<n(t) \leq n(\bar{t})$ and $\displaystyle \lim_{t \to \infty}n(t)=0$,
    where $\bar{t}$ is the unique solution of $L(t)=\frac{N+M}{2}$;
\item[(2b)] If $L_{0} \in\left(\frac{N+M}{2}, M\right)$, then $0<n(t) \leq n\left(t_{0}\right)$ and $\displaystyle \lim_{t \to \infty}n(t)=0$;
\item[(3)] If $L_{0} \in(M, \infty)$, then $n\left(t_{0}\right)<n(t)<0$, $n^{\prime}(t)>0$, and $\displaystyle \lim_{t \to \infty}n(t)=0$;
\item[(4)] If $L_{0} \in\{N, M\}$, then $n(t) \equiv 0$.
\end{itemize}
\label{n(t)-bound-allee effect}
\end{remark}
 In any case in Remark \ref{n(t)-bound-allee effect}, the absolute value of the population growth rate $n(t)$ is bounded with respect to $t$, i.e. there exists $\eta>0$ with
\begin{equation}
|n(t)| \leq \eta \text { for all } t \geq t_{0}   \label{|n(t)|< eta} .
\end{equation}

To analyze the optimal control problem of maximizing (\ref{eq:utility}) subject to (\ref{eq:cap-new}), we apply the Pontryagin's maximum principle \cite{10}. The corresponding Hamiltonian function is 
\begin{equation}
    H = e^{-\rho t}u(c) + \lambda \left( f(k) - c - (\delta + n)k \right),
\end{equation}
where $\lambda$ is a costate variable depending on $t$. The necessary conditions for an optimal solution include:
\begin{equation}
\begin{aligned}
\frac{\partial \mathcal{H}}{\partial c}  =0, \quad
\frac{\partial \mathcal{H}}{\partial k}+\dot{\lambda}  =0, \quad
\frac{\partial \mathcal{H}}{\partial \lambda} - \dot{k} = 0.
\end{aligned}
\end{equation}
With the transverse condition (\cite{2} equation (8)) that guarantees the  optimality for \(c(t)\), becomes
\begin{equation}
\lim_{t \to \infty} e^{-\rho t}c^{-1/\sigma}(t) k(t) = 0.
\label{Trans-c}
\end{equation}

From here, we derive the following system of ordinary differential equations:
\begin{equation}
\label{eq:main-ode}
\begin{aligned}
& \dot{k}(t)=f\left(k(t)\right)-\left(\delta+r \left(1 - \frac{L(t)}{M}\right) \left(\frac{L(t)}{N} - 1\right)\right) \cdot k(t)-c(t), \\
& \dot{c}(t)=\sigma\cdot c(t)\cdot\left[f^{\prime}\left(k(t)\right)-\rho-\delta-r \left(1 - \frac{L(t)}{M}\right) \left(\frac{L(t)}{N} - 1\right)\right], \\
& \dot{L}(t) = r \cdot L(t) \cdot \left(1 - \frac{L(t)}{M}\right) \cdot\left(\frac{L(t)}{N} - 1\right).
\end{aligned}
\end{equation}
These equations provide insight into the evolution of capital, consumption, and population in case of optimal saving. We see that this
system of equations can be divided into two subsystems, one is the independent
subsystem of population growth, and the other is the dynamical system of capital and consumption related to population.

\section{Analysis of the model}
We analyze the behavior of the solution of (\ref{eq:main-ode})
mainly by adopting the method of comparison based on the following result from the literature.

\begin{theorem}[Comparison theorem, \cite{5}]
Let $u^{1}(t), u^{2}(t)$ be solutions of $$\dot{u}(t)=\varphi_{1}\left(t, u(t)\right), \quad \dot{u}(t)=\varphi_{2}\left(t, u(t)\right),$$ respectively, with the same initial condition $u^{1}(0)=u^{2}(0)$. If $\varphi_{1}\left(t, u\right) \leq \varphi_{2}\left(t, u\right)$ for all $\left(t, u\right)$, then $u^{1}(t) \leq u^{2}(t)$ for all $t$.
\label{T1}
\end{theorem}

\subsection{Bounds for the capital}
\label{ss:bound-capital}
For what follows, it is convenient to consider solutions \( k^1(t), k^2(t) \) of
\[
\dot{k}(t) = -(\delta + n(t))k(t), \quad \text{and} \quad \dot{k}(t) = f(k(t)) - (\delta + n(t))k(t),
\]
respectively, with the initial value \( k^1(t_0) = k^2(t_0) = k_0 \geq 0\). In virtue of Theorem \ref{T1}, the solution \( k(t) \) of (\ref{eq:main-ode}) satisfies
\[
k^1(t) \leq k(t) \leq k^2(t).
\]

\subsubsection{Upper bound}
\label{ssec:u-b}
Since here the output is used solely for the capital reproduction, i.e. as if $c(t) = 0$, the dynamics for $k^2(t)$ reduces to the Solow-Swan model with Allee effect studied in \cite{1}:
\begin{equation}
\label{eq:solow-c-zero}
    \dot{k}(t) = f(k(t)) - (\delta + n(t))k(t).
\end{equation}
From there we know that
 stability for (\ref{eq:solow-c-zero}) is guaranteed if and only if $\delta> r$. In this case, if $n(t) \rightarrow n_\infty$ for $t \rightarrow \infty$, then 
\[
    \lim_{t \to \infty} k^2(t) = {k}_{n_\infty}^{*},
\]
where
${k}_{n_\infty}^*$ is the equilibrium of the classical Solow-Swan model \cite{6,7}:
\begin{equation}
    \dot{k}(t) = f(k(t)) - (\delta + n_\infty)k(t),
    \label{solow-model-equi}
\end{equation}
and ${k}_{n_\infty}(t)$ denotes its solution with ${k}_{n_\infty}(t_0)=k_0$.
In particular, we obtain the following results under $\delta>r$:
\label{upper-bound}
\begin{itemize}
\item \textbf{Case $L_{0} \in(0, N)$.}
\end{itemize}
The population growth rate approaches $n_\infty = -r$. Hence, for any $k_{0}>0$ it holds:
\[
\lim _{t \rightarrow \infty} k^{2}\left(t\right)=k^{*}_{-r}
\mbox{ and }
k^{2}\left(t\right) \leq k_{-r}\left(t\right).
\]
\begin{itemize}
\item \textbf{Case $L_{0} \in(N, M)$.}
\end{itemize}
The population growth rate approaches $n_\infty = 0$. Hence, for any $k_{0}>0$ it holds:
\[
\lim _{t \rightarrow \infty} k^{2}\left(t\right)=k^{*}_{0}
\mbox{ and }
k^{2}\left(t\right) \leq k_{0}\left(t\right).
\]
\begin{itemize}
\item \textbf{Case $L_{0} \in(M, \infty)$.}
\end{itemize}
The population growth rate approaches $n_\infty = 0$. Hence, for any $k_{0}>0$ it holds:
\[
\lim _{t \rightarrow \infty} k^{2}\left(t\right)=k^{*}_{0}
\mbox{ and }
k^{2}\left(t\right) \leq k_{n(t_0)}\left(t\right).
\]

\subsubsection{Lower bound}

Here, the capital per labour $k^1(t)$ decreases over time and eventually reaches zero because the capital per labour is continually not reinvested into production. The entire output of this production is consumed, i.e. as if \( c(t) = f(k(t)) \). We can explicitly solve
\begin{equation}
\label{eq:solow-c-f}
    \dot{k}(t) =  - (\delta + n(t))k(t),
\end{equation}
and obtain
  \begin{equation}
      k^{1}(t) = k^{1}(0) \exp\left( -\int_0^t (\delta + n(\tau)) \, d\tau \right).
      \label{eq:lower-r}
  \end{equation}
In view of (\ref{eq:lower-r}), we get for the lower bound under $\delta>r$:
\begin{equation}
    \lim _{t \rightarrow \infty} k^{1}\left(t\right)=0.
    \label{k1-stability-}
\end{equation}


\subsection{Bounds for the consumption}

Now, we turn our attention to the upper and lower bounds for consumption.

\subsubsection{Upper bound}
From  (\ref{S(t)}) and (\ref{I(t)}) we know that \( Y = I + C \). If we divide all variables by the quantity of labour \( L \),
\[
\frac{Y(t)}{L(t)}=\frac{I(t)}{L(t)}+\frac{C(t)}{L(t)},
\]
we can obtain the lowercase version 
\begin{equation}
    f(k(t)) = i(t) + c(t), 
\label{f(t)help}
\end{equation}
where the investment per unit of labour is set as
\[
i(t) = \frac{I(t)}{L(t)}.
\]
From (\ref{f(t)help}) together with \(I=sY \geq 0 \), we conclude that \(c(t) \leq f(k(t))\).
Due to the monotonicity of $f$ and the upper bounds for the capital from Section \ref{ssec:u-b}, we obtain  for any initial consumption $c(t_0)=c_0>0$ under $\delta>r$: 
\begin{itemize}
\item \textbf{Case $L_{0} \in(0, N)$.}
\end{itemize}
\[
c(t)  \leq f(k_{-r}^2(t)).
\]
\begin{itemize}
\item \textbf{Case $L_{0} \in(N, M)$.}
\end{itemize}
\[
c(t) \leq  f(k_{0}^2(t)).
\]
\begin{itemize}
\item \textbf{Case $L_{0} \in(M, \infty)$.}
\end{itemize}
\[
c(t) \leq f(k_{n(t_0)}^2(t)).
\]
\subsubsection{Lower bound}
We derive the lower bound from
the corresponding equation in (\ref{eq:main-ode}): 
$$\frac{\dot{c}(t)}{\sigma c(t)} =f^{\prime}\left(k(t)\right)-\delta-\rho-n(t).$$ 
Its solution has the following form
$$
c(t)=c_{0} L_0^{-\sigma}   e^{\displaystyle\sigma \int_{0}^{t} f^{\prime}\left(k(\tau)\right)d \tau} e^{-\sigma(\delta+\rho)  t}L^{-\sigma}(t).
$$
Next, we use the fact that $f$ is monotonically increasing and concave. 
The following cases are possible:
\begin{itemize}
    \item if $k_{0} < k^{*}_{n_\infty}$, then we use $k(t) \leq k^{*}_{n_\infty}$ to deduce
the lower bound as 
$$
c(t) \geq 
c_{0}    e^{\displaystyle\sigma(B-\delta-\rho) t} L^{-\sigma}(t).
$$
with
$B=f^{\prime}\left(k^{*}_{n_\infty}\right)$.

    \item if $k_{0}\geq k^{*}_{n_\infty}$, then analogously we deduce
the lower bound as 
$$
c(t) \geq 
c_{0}     e^{\displaystyle\sigma(B-\delta-\rho) t} L^{-\sigma}(t)
$$
with
$B=f^{\prime}\left(k_0\right)$.
\end{itemize}
Now, we ensure that the derived lower bound of \( c(t) \) is not blowing up.
\begin{lemma}
    If $n(t)$ is bounded by $|n(t)| \leq \eta$ for all $t\geq 0$, and $\delta > \eta$ holds,
    then
    $$
  \lim_{t \to +\infty}  c_{0} e^{\displaystyle\sigma(B-\delta-\rho) t} L^{-\sigma}(t)=0,
    $$
where $B$ is set as above.
    \label{c(t)-lowerbound-infty}
\end{lemma}
\begin{proof}
First, we rewrite for the lower bound:
$$
c(t) \geq c_{0} e^{\displaystyle \sigma(B-\delta-\rho) t} L^{-\sigma}_0 e^{\displaystyle-\sigma\int^t_0 n(\tau)d \tau}.
$$
   We need to distinguish two cases to prove the assertion:
    \begin{itemize}
        \item Case $0 \leq n(t) \leq \eta$ for all $t \geq 0$.
\end{itemize}

Then, the last term from above is bounded:
\[
L^{-\sigma}(t) = L^{-\sigma}_0 e^{\displaystyle-\sigma\int^t_0 n(\tau)d \tau} \leq L^{-\sigma}_0.
\]
Further, by the assumption that $f$ is strictly concave, we have
$$
0 < f(k(t)) -k(t)f^{\prime}(k(t)).
$$
From here we obtain
\begin{equation}
\frac{f\left(k(t)\right)}{k(t)} > f^{\prime}\left(k(t)\right)
 \label{f(k)/k >= f'(k)}.
\end{equation}
If $k_0 < k^{*}_{n_\infty}$,  we take the limit in  (\ref{solow-model-equi}). 
Since $\displaystyle \lim_{t\to +\infty }n(t)=0$, cf. Remark \ref{n(t)-bound-allee effect}, and, hence, $n_\infty=0$, we obtain:
\[
0 = f(k^*_{n_{\infty}}) - (\delta + 0)k^*_{n_{\infty}}.
\]
By means of (\ref{f(k)/k >= f'(k)}), we then deduce:
$$
\delta=\frac{f\left(k^{*}_{n_\infty}\right)}{k^{*}_{n_\infty}}> f^{\prime}\left(k^{*}_{n_\infty}\right) = B.
$$
If $ k^{*}_{n_\infty} \leq k_{0}$,
we have again due to (\ref{f(k)/k >= f'(k)}):
\begin{equation}
 \left(\frac{f(k)}{k} \right)^\prime =
\frac{f^\prime(k)k-f(k)}{k^2}  {<} 0.
\label{f(k)/k-monoton}
\end{equation}
It means that $\frac{f(k)}{k}$ is a decreasing function. This implies in particular that 
$$
\delta=\frac{f\left(k^{*}_{n_\infty}\right)}{k^{*}_{n_\infty}}  \geq \frac{f\left(k_0\right)}{k_0} > f^\prime(k_0)=B.
$$
Overall, in both cases we proved $\delta \geq B$.
Then, by $\rho > 0$ we have 
$B-\delta-\rho < 0$. Hence, we obtain
 $$
  \lim_{t \to +\infty}  c_{0} e^{\displaystyle\sigma(B-\delta-\rho) t} L^{-\sigma}(t)=0.
$$

\begin{itemize}
    \item Case $-\eta \leq n(t) \leq 0$. 
\end{itemize}
We set
$\bar{\delta}=\delta-\eta$ and $ \bar{n}(t)=n(t)+\eta$.
Hence,
$\bar{\delta}>0$ and $0\leq\bar{n}(t)\leq \eta$.
The lower bound above now equals to
$$
c_{0} e^{\displaystyle\sigma(B-\bar{\delta}-\rho) t} L^{-\sigma}_0 e^{\displaystyle-\sigma\int^t_0 \bar{n}(\tau) d \tau}.
$$
Because of $0\leq\bar{n}(t)\leq \eta$, we can apply the first case to obtain the assertion. 
\end{proof}

Finally, we are ready to analyze the limiting behavior of the lower bound of $c(t)$ from above:

\begin{itemize}    
\item \textbf{Case $L_{0} \in(0, N)$.}
\end{itemize}
Here, we set $\eta=r$, cf. Remark \ref{n(t)-bound-allee effect}, and by means of Lemma \ref{c(t)-lowerbound-infty}, we immediately obtain
\[ \lim_{t \to +\infty}  c_{0} e^{\displaystyle\sigma(B-\delta-\rho) t} L^{-\sigma}(t)=0.
\] 
\begin{itemize}    
\item \textbf{Case $L_{0} \in(N, M)$.}
\end{itemize}
Here, the population growth rate approaches $n_\infty = 0$. Then, we have $$\lim_{t \to +\infty}L^{-\sigma}(t)\leq L_0^{-\sigma}. $$
By using Lemma \ref{c(t)-lowerbound-infty}, we also have $ B-\delta-\rho < 0 $. Overall, 
$$
  \lim_{t \to +\infty}  c_{0} e^{\displaystyle\sigma(B-\delta-\rho) t} L^{-\sigma}(t)=0.
$$
\begin{itemize}    
\item \textbf{Case $L_{0} \in(M, \infty)$.}
\end{itemize}
The population growth rate has the following properties, see Remark \ref{n(t)-bound-allee effect} from above:
$$
n(t)<0,\quad n^{\prime}(t)>0 ,\quad \lim _{t \rightarrow \infty} n(t)=0.
$$
 We just set $\eta=n(\hat{t})$ with $\hat{t}$ is taken sufficiently large to guarantee $\delta+n(\hat{t})>0$. By means of Lemma \ref{c(t)-lowerbound-infty} again, we immediately obtain
\[ \lim_{t \to +\infty}  c_{0} e^{\displaystyle\sigma(B-\delta-\rho) t} L^{-\sigma}(t)=0.
\]

\subsection{Limiting behavior}
In this section, we shall analyze the limiting behavior of the solution in (\ref{eq:main-ode}). First, by assumption on the
population growth rate due to the Allee effect (\ref{n(t)-bound-allee effect}), we recall that
$$
\lim_{t \to +\infty}n(t)=n_\infty.
$$
However, the convergence of \( k(t) \) cannot be guaranteed in general.
In the following discussion, we assume that \( k(t) \) is convergent, and
and then we prove the following result on the long-run behavior of \(k(t)\) and \(c(t)\) at the equilibrium.

\begin{theorem}
Assume that $\displaystyle\lim _{t \rightarrow \infty} k(t)=k_{\infty}$ exists and $\delta>r$. Then, $\displaystyle \lim _{t \rightarrow \infty} c(t)=c_{\infty}$ also exists, and exactly one the following situations is possible:
\begin{itemize}
    \item[(I)] $k_{\infty}=k_{n_\infty}^{*}, c_{\infty}=0$,
     \item[(II)] $k_{\infty}=\left(f^{\prime}\right)^{-1}\left(\rho+\delta+n_{\infty}\right), c_{\infty}=f\left(k_{\infty}\right)-\left(\delta+n_{\infty}\right) k_{\infty}$.
 \end{itemize}
 Moreover, if $c_0 >0$ we have case (II).
\label{T-uni}
\end{theorem}

\begin{proof}
By Section \ref{ss:bound-capital} on the lower and upper bounds for the capital, we have 
$$k^1(t) \leq k(t) \leq k^2(t).$$
Then, it holds:
$$0 \leq k_{\infty} \leq k_{n_\infty}^{*}.$$ 
Equality on the left holds only when \( c_0 = f(k_0) \), in which case both \( k(t) \) and \( c(t) \) converge to zero.  
Since we do not consider this degenerate case (we assume \( f(k_0) > c_0 > 0 \), with \( k_0 > 0 \)), the equality is not attained.
$$0 < k_{\infty} \leq k_{n_\infty}^{*}.$$ 
By taking $t$ to infinity in (\ref{eq:main-ode}), $c(t)$ converges to
\begin{equation}
c_{\infty}=f\left(k_{\infty}\right)-\left(\delta+n_{\infty}\right) k_{\infty}.
\label{c_infnity}
\end{equation}
Additionally, (\ref{eq:main-ode}) also provides in the limit:
\begin{equation}
0= \sigma c_{\infty}\left(f^{\prime}\left(k_{\infty}\right)-\rho-\delta-n_{\infty}\right).  
\end{equation}
From here it follows:

\begin{equation}
 \left[f\left(k_{\infty}\right)-\left(\delta+n_{\infty}\right) k_{\infty}\right]\left[f^{\prime}\left(k_{\infty}\right)-\rho-\delta-n_{\infty}\right]=0.
 \label{two-part-both-not-zero}
\end{equation}
First, notice that both terms on the left-hand side of this equation cannot both vanish. Otherwise, we would have:
$$\frac{f\left(k_{\infty}\right)}{k_{\infty}} =\delta+n_{\infty}, \quad f^{\prime}\left(k_{\infty}\right)=\rho+\delta+n_{\infty}.$$ Since $f$ is strictly concave by assumption on $F$, it holds
$$\frac{f\left(k_{\infty}\right)}{k_{\infty}} -f^{\prime}\left(k_{\infty}\right) > 0.$$ 
Altogether, we would have $\rho < 0$, which contradicts to our assumption on $\rho$. 
Now, let us discuss the corresponding cases: 
\begin{itemize}
    \item Case (I) with $\left[f\left(k_{\infty}\right)-\left(\delta+n_{\infty}\right) k_{\infty}\right]=0.$
\end{itemize}
We know that $k^2(t)$ is the solution of Solow model
$$
\dot{k}(t)=f(k(t))-(\delta+n(t))k(t).
$$
By recalling $$
k_{n_\infty}^{*}=\lim_{t \to +\infty}k^2(t),
$$
we then obtain:
$$
0 = f(k_{n_\infty}^{*})-(\delta+n_\infty)k_{n_\infty}^{*}.
$$
The strict concavity of the function $f$ along with the Inada conditions ensure that there exists the unique solution for $$\frac{f\left(k\right)}{ k}=\delta+n_{\infty}.$$ 
This is since the function $\displaystyle\frac{f\left(k\right)}{k}$ monotonically decreasing, cf. (\ref{f(k)/k-monoton}), and $\delta+n_{\infty} > 0$ due to the assumption $\delta>r$. Overall, we get then $k_{\infty}=k_{n_\infty}^{*}$. In view of (\ref{c_infnity}), we have $c_\infty=0$.
\begin{itemize}
    \item Case (II) with $\left[f^{\prime}\left(k_{\infty}\right)-\rho-\delta-n_{\infty}\right]=0$.
\end{itemize}
We rewrite the equation into
$$
f^{\prime}\left(k_\infty\right)=\rho+\delta+n_{\infty}.
$$
Notice that $f^{\prime}$ admits the inverse since it is a strictly monotonic function due to $f^{\prime}> 0$. Then, we obtain
$$
k_{\infty}=\left(f^{\prime}\right)^{-1}\left(\rho+\delta+n_{\infty}\right).
$$
Together with the equation (\ref{c_infnity}), it holds then:
$$
c_\infty=f\left(k_{\infty}\right)-\left(\delta+n_{\infty}\right) k_{\infty}.
$$

Now, let us show that if $c_0 >0$ we have case (II). This will be deduced from 
the transversality condition (\ref{Trans-c}):
\[
\lim_{t \to \infty} e^{-\rho t}c^{-1/\sigma}(t) k(t) = 0.
\]
By using the explicit formula for consumption
\[
c(t) = c_0 \cdot e^{\displaystyle \sigma \int_0^t \left[ f'(k(s)) - \delta - n(s) - \rho \right] ds},
\]
we obtain after simplifications:
\[
\lim_{t \to \infty}  c_0^{-1/\sigma} \cdot k(t) \cdot e^{\displaystyle - \int_0^t \left[ f'(k(s)) - \delta - n(s) \right] ds }=0.
\]
Since $k(t)$ converges towards $k_\infty >0$, as can be easily seen from cases (I) or (II), and $c_0 >0$, the integral above has to blow up:
\begin{equation}
\label{eq:blow-up}
\int_0^t \left[ f'(k(s)) - \delta - n(s) \right] ds \to +\infty.
\end{equation}
This condition can be further analyzed. For that, assume we are in Case (I), i.e. $k_\infty=k_{n_\infty}^*$. Since $k_{n_\infty}^*$ is a steady-state solution of the standard Solow model with $n_\infty$, we have:
\[
f(k_{n_\infty}^*) - (\delta + n_\infty) k_{n_\infty}^* = 0.
\]
Due to the strict concavity of $f$, we obtain:
\[
\delta + n_\infty = \frac{f(k_{n_\infty}^*)}{k_{n_\infty}^*} > f'\left(k_{n_\infty}^*\right).
\]
Now, since $\displaystyle \lim_{t \to \infty} n(t) = n_\infty$, for sufficiently large $t$, $n(t)$ lies within a neighborhood of $n_\infty$, say $(n_\infty - \varepsilon_1, n_\infty + \varepsilon_1)$ with $\varepsilon_1 >0$. Similarly, since $\displaystyle \lim_{t \rightarrow \infty}k(t) = k_{n_\infty}^*$ and $f'$ is continuous, it holds for sufficiently large $t$:
\[
f'(k(t)) \in \left(f'(k_{n_\infty}^*) - \varepsilon_2,\ f'(k_{n_\infty}^*) + \varepsilon_2\right),
\]
where $\varepsilon_2 >0$ is arbitrary.
Because of $\delta + n_\infty > f'(k_{n_\infty}^*)$, we can choose $\varepsilon_3 > 0$ such that:
\[
\delta + n_\infty - \varepsilon_3 > f'(k_{n_\infty}^*) + \varepsilon_3.
\]
Therefore, by setting $\varepsilon = \min\left\{\varepsilon_1, \varepsilon_2, \varepsilon_3\right\}$, it holds for $t > T$:
\[
f'(k(t)) \leq f'(k_{n_\infty}^*) + \varepsilon < \delta + n_\infty - \varepsilon \leq \delta + n(t),
\]
which implies:
\[
f'(k(t)) - \delta - n(t) < 0.
\]
Hence, it holds for the integral
\[
\int_T^\infty \left[ f'(k(s)) - \delta - n(s) \right] ds \leq 0.
\]
This provides a contradiction to the blow up in (\ref{eq:blow-up}), since
\[
\int_0^t \left[ f'(k(s)) - \delta - n(s) \right] ds \leq \int_0^T \left[ f'(k(s)) - \delta - n(s) \right] ds,
\]
where on the right-hand side we have a constant.
Overall, case (I) is not possible, and, hence, case (II) happens.

\end{proof}

Theorem \ref{T-uni} allows to compare the equilibrium capital and consumption per labour corresponding to the Allee effect. As above, we distinguish Cases (I) and (II), by additionally assuming $\delta >r$.

\begin{remark}[Case (I)]
  Here, we start from the zero consumption level with $c_0=0$.
  If $L_0 < N$, we have for the equilibrium capital per capita $k_{\infty,-r}=k_{-r}^{*}$, since $n_\infty=-r$. If $L_0 \geq N$, it holds
$k_{\infty,0}=k_{0}^{*}$, since $n_\infty=0$ here, see also Remark \ref{n(t)-bound-allee effect}. Note that, due to Theorem \ref{T1}, $k_{-r}^{*} > k_{0}^{*}$, meaning that the capital per labour stabilizes at a higher level in case of diminishing rather than of growing population, i.e. $k_{\infty,-r} > k_{\infty,0}$. This is in accordance with the results from \cite{1}. In both cases, the consumption per labour approaches zero, i.e. $c_\infty=0$. \qed
\end{remark}

Whereas Case (I) turns out to be more or less known from the previous literature, Case (II) sheds light on the economic interpretations of the Ramsey model with Allee effect, especially, with respect to its consumption part.

\begin{remark}[Case (II)]
  Now, we start from the nonzero consumption level with $c_0>0$.
  If $L_0 < N$, we have $n_\infty=-r$ and, thus, for the equilibria of (\ref{eq:main-ode})
  $$
k_{\infty,-r}= \left(f^{\prime}\right)^{-1}\left(\rho+\delta-r\right), \quad c_{\infty,-r}=f\left(k_{\infty,-r}\right)-\left(\delta-r\right) k_{\infty,-r}.
$$
If $L_0 \geq N$, we have $n_\infty=0$ and, thus, for the equilibria of (\ref{eq:main-ode})
  $$
k_{\infty,0}= \left(f^{\prime}\right)^{-1}\left(\rho+\delta\right), \quad c_{\infty,-r}=f\left(k_{\infty,0}\right)-\delta k_{\infty,0}.
$$
First, we compare the equilibrium capital per labour in these two cases. Since
$f^\prime$ is monotonically decreasing, so is its inverse. Therefore, we conclude:
$$
\left(f^{\prime}\right)^{-1}\left(\rho+\delta-r\right) > \left(f^{\prime}\right)^{-1}\left(\rho+\delta\right).
$$
Again, the capital per labour
stabilizes at a higher level in case of diminishing rather than of growing population, i.e. $k_{\infty,-r}>k_{\infty,0}$. How are the corresponding equilibrium consumptions per labour related? For this, we define the difference
$$
D_c=c_{\infty,-r}-c_{\infty,0}=f(k_{\infty,-r})-f(k_{\infty,0})-(\delta-r)k_{\infty,-r}+\delta k_{\infty,0}.
$$
By using the mean value theorem, there exists $\xi \in \left[k_{\infty,0},k_{\infty,-r}\right]$ such that
$$
D_c=f^{\prime}(\xi)(k_{\infty,-r}-k_{\infty,0})+\delta k_{\infty,0}-(\delta-r)k_{\infty,-r}.
$$
Since $f^\prime$ is monotonically decreasing, we obtain
$$
D_c\geq f^{\prime}(k_{\infty,-r})(k_{\infty,-r}-k_{\infty,0})+\delta k_{\infty,0}-(\delta-r)k_{\infty,-r}.
$$
Recall that $k_{\infty,-r}=\left(f^{\prime}\right)^{-1}\left(\rho+\delta-r\right)$ and substitute this into the previous formula: 
$$
D_c \geq 
\rho \left(k_{\infty,-r} - k_{\infty,0}\right)+r k_{\infty,0}.
$$
Due to $k_{\infty,-r} > k_{\infty,0}$,
the latter provides $D_c >0$.
Overall, we conclude that the equilibrium consumption per labour stabilizes at a higher level in case of diminishing rather than growing population, i.e. 
$c_{\infty,-r} > c_{\infty,0}$. \qed
\end{remark}

Finally, we investigate the limiting behavior of the consumption-to-capital ratio for the Ramsey model with Allee effect, i.e.
$$
x(t)=\frac{c(t)}{k(t)}.
$$
This is to compare in Case (II) how consumption and capital relate to each other, depending on whether the population grows or diminishes. 
For this, we set the share of capital to production as auxiliary variable
$$
z(t)=\frac{k(t)}{f\left(k(t)\right)}.
$$
As elaborated in \cite{2}, we obtain from (\ref{eq:main-ode}):
\begin{equation}
\begin{aligned}
& \frac{\dot{z}(t)}{z(t)} = \left(1 - f'\left(k(t)\right) z(t)\right) \left[ z(t)^{-1} - \left(\delta + n(t)\right) - x(t) \right], \\
& \frac{\dot{x}(t)}{x(t)} = \left[ \sigma f'\left(k(t)\right) - z(t)^{-1} \right] + (1 - \sigma)\left( \delta + n(t) \right) - \sigma \rho + x(t).
\end{aligned}
\label{eq:ratio-ode}    
\end{equation}

For this transformed version of the Ramsey system, we similarly investigate its stability by identifying the steady states of (\ref{eq:ratio-ode}). According to the previous results and Theorem \ref{T-uni}, assuming the existence of \( k_{\infty} \), we also know that \( c_{\infty} \) exists. Consequently, under the assumption \( k_0 > 0 \) and \( c_0 > 0 \), i.e. in Case (II), we have that \( x_{\infty} = \frac{c_{\infty}}{k_{\infty}} \) also exists, and \( z_{\infty} = \frac{k_{\infty}}{f(k_{\infty})} \) exists as well. Next result shows that in the long-run the consumption-to-capital ratio cannot fall below the individual discounting.

\begin{corollary}[Consumption-to-capital ratio]
    Under assumptions of Theorem \ref{T-uni} with $c_0 >0$, the ratio of consumption and capital at the steady state exceeds the discount rate:
\[
x_\infty > \rho.
\]
\end{corollary}

\begin{proof}
Since we already know from the Theorem \ref{T-uni} that \( z_\infty \neq 0 \) and \( x_\infty \neq 0 \), we get from (\ref{eq:ratio-ode}):
\[
\begin{aligned}
& 0 = \left(1 - f'\left(k_\infty\right) z_\infty\right) \left[ z_\infty^{-1} - \left(\delta + n_\infty\right) - x_\infty \right], \\
& 0 = \left[ \sigma f'\left(k_\infty\right) - z_\infty^{-1} \right] + (1 - \sigma)\left( \delta + n_\infty \right) - \sigma \rho + x_\infty.
\end{aligned}
\]
In the first equation, we analyze the product of two terms. Due to the concavity of  \( f \), we have as in (\ref{f(k)/k >= f'(k)}):
\[
\frac{f(k_\infty)}{k_\infty} > f'\left(k_\infty\right),
\]
which implies:
\[
z_\infty = \frac{k_\infty}{f(k_\infty)} < \frac{1}{f'\left(k_\infty\right)} \quad \Rightarrow \quad f'\left(k_\infty\right) z_\infty < 1.
\]
Therefore, the only way the first equation can hold is if the second term vanishes:
\[
z_\infty^{-1} - \left(\delta + n_\infty\right) - x_\infty = 0.
\]
Substituting this into the second equation yields:
\[
\sigma f'\left(k_\infty\right) - \left( \delta + n_\infty + x_\infty \right) + (1 - \sigma)\left( \delta + n_\infty \right) - \sigma \rho + x_\infty = 0.
\]
The latter simplifies to
\[
f'\left(k_\infty\right) = \delta + n_\infty + \rho,
\]
which is by the way also known from Theorem~\ref{T-uni}. Overall, the formula for the steady state for the $x$-part follows:
\[
\begin{aligned}
& x_\infty = \frac{1}{z_\infty} - \delta - n_\infty
\end{aligned}
\]
or, equivalently, 
\[
x_\infty = \frac{f(k_\infty)}{k_\infty} - \delta - n_\infty.
\]
Using the concavity of  \( f \) once again, we deduce that
\[
x_\infty > f'\left(k_\infty\right) - \delta - n_\infty = \rho.
\]
\end{proof}


Now, we start from a strictly positive consumption level \( c_0 > 0 \), cf. Case (II). Our goal is to compare the steady-state consumption-to-capital ratio \( x_\infty = \frac{c_\infty}{k_\infty}\) under different initial labour values \( L_0 \), by distinguishing two separate cases:

\begin{itemize}
    \item In case \( L_0 < N \)
population growth rate becomes negative in the long run, i.e. \( n_\infty = -r < 0 \). From the equilibrium condition of the original system~(\ref{eq:main-ode}), the corresponding steady-state values satisfy:
\[
k_{\infty,-r} = \left(f'\right)^{-1}(\rho + \delta - r), \quad
c_{\infty,-r} = f(k_{\infty,-r}) - (\delta - r)k_{\infty,-r}.
\]
    \item In case \( L_0 \geq N \), the population growth stabilizes at zero, i.e., \( n_\infty = 0 \). The equilibrium condition then yields:
\[
k_{\infty,0} = \left(f'\right)^{-1}(\rho + \delta), \quad
c_{\infty,0} = f(k_{\infty,0}) - \delta k_{\infty,0}.
\]
\end{itemize}

In order to compare the steady-state consumption-to-capital ratios in both cases, we define the difference:
\[
D_x := x_{\infty,-r} - x_{\infty,0} = \frac{c_{\infty,-r}}{k_{\infty,-r}} - \frac{c_{\infty,0}}{k_{\infty,0}}.
\]
By using formulas above, we obtain
\[
D_x = \left[ \frac{f(k_{\infty,-r})}{k_{\infty,-r}}-f'(k_{\infty,-r})\right] - \left[\frac{f(k_{\infty,0})}{k_{\infty,0}}-f'(k_{\infty,0}) \right].
\]
Define
\[
h(k):=\frac{f(k)}{k}-f'(k),\qquad k>0.
\]
A direct differentiation yields
\[
h'(k)
=\frac{k f'(k)-f(k)}{k^2}-f''(k).
\]
The two terms on the right-hand side have opposite signs in general. In fact, the strict concavity of $f$ gives \(k f'(k)< f(k)\), so the first term is negative, while \(f''(k)<0\) makes the second term positive. Hence, the sign of \(h'(k)\) depends on the particular choice of $f$. In order to illustrate this, we consider:
\begin{itemize}
  \item Cobb--Douglas production function \(f(k)=k^\alpha\) with \(0<\alpha<1\). Then,
  \[
  h(k)=\frac{k^\alpha}{k}-\alpha k^{\alpha-1}=(1-\alpha)k^{\alpha-1}.
  \]
  Since \(\alpha-1<0\), we have \(h'(k)<0\). In this case, \(h\) is strictly decreasing, and, thus, \(D_x<0\).
\item Logarithmic production function $f(k)=\ln(1+k)$.
We have
\[
h(k) = \frac{\ln(1+k)}{k} - \frac{1}{1+k}, \qquad
h'(k) = \frac{\tfrac{k}{1+k} - \ln(1+k)}{k^{2}} + \frac{1}{(1+k)^{2}}.
\]
The derivative $h'(k)$ has the unique zero $k^*_{\log} \approx 2.16258$.
Moreover, $h$ is increasing on $(0,k^*_{\log})$ and decreasing on $(k^*_{\log},\infty)$. Thus, the sign of \(D_x\) becomes dependent on whether $k_{\infty,-r}$ and $k_{\infty,0}$ lie in one interval or another.
E.g., letting the  parameters $\delta = 0.075$, $\rho = 0.02$, $r = 0.025$, we have
    \[
    k_{\infty,-r} = 0.042857, \quad k_{\infty,0} = 0.005263, \quad D_x \approx -0.00123.
    \]
    For parameters $\delta = 0.75$, $\rho = 0.2$, $r = 0.25$ it holds:
    \[
    k_{\infty,-r} = 0.428571, \quad k_{\infty,0} = 0.052632, \quad D_x \approx 0.107669.
    \]
\item Constant absolute risk-aversion production function $f(k)=1-e^{-k}$. Here,
\[
h(k) = \frac{1-e^{-k}}{k} - e^{-k}, \qquad
h'(k) = \frac{k e^{-k} - 1 + e^{-k}}{k^{2}} +  e^{-k}.
\]
Numerical computations yield that
$k^*_{\text{CARA}} \approx 1.79328$
is the unique zero of $h'(k)$.
Moreover, $h$ is increasing on $(0,k^*_{\text{CARA}})$ and decreasing on $(k^*_{\text{CARA}},\infty)$. Again, the sign of \(D_x\) depends on $k_{\infty,-r}$ and $k_{\infty,0}$.
E.g., for the parameters $\delta = 0.075$, $\rho = 0.02$, $r = 0.025$ we have
    \[
    k_{\infty,-r} = 0.035667, \quad k_{\infty,0} = 0.005129, \quad D_x \approx -0.00145.
    \]
If we change the parameters into $\delta = 0.75$, $\rho = 0.2$, $r = 0.25$, the situation becomes different:
    \[
    k_{\infty,-r} = 0.356675, \quad k_{\infty,0} = 0.051293, \quad D_x \approx 0.116316.
    \]
\end{itemize}

\section{Numerical simulations}
 We choose the well-known CES production function:

\begin{equation}
F(K,L) =  \left[ \alpha K^{\tau} + (1 - \alpha) L^{\tau} \right]^{\frac{1}{\tau}},
\end{equation}
where $\alpha$ is the capital distribution parameter with $0 < \alpha < 1$, and  $\tau$ is the substitution elasticity parameter with $-\infty < \tau < 1$.
Then, the output per unit of labour is
\begin{equation}
f(k)=\frac{F(K,L)}{L}=\left( \alpha k^{\tau} + (1 - \alpha) \right)^{\frac{1}{\tau}}.
\label{CES-f}
\end{equation}
Then, we can obtain its derivative:
\begin{equation}
 f^\prime(k)=\alpha \left( \alpha k^{\tau} + (1 - \alpha) \right)^{\frac{1}{\tau} - 1} k^{\tau - 1}.   
\end{equation}
 The parameters related to capital are chosen as follows:
$$
\alpha=0.3, \quad \delta=0.075 , \quad \tau=0.01,
\quad
\rho=0.02,  \quad \sigma=0.01.
$$
The parameters determining the population dynamics are set as
$$
r=0.025, \quad N=1, \quad M=2.
$$
The behavior of capital per labour under the Allee effect aligns as expected in Case (II), see Figure \ref{k-comparing}. 
\begin{figure}[htp]
    \centering
    \includegraphics[width=0.8\textwidth,height=0.6\textwidth]{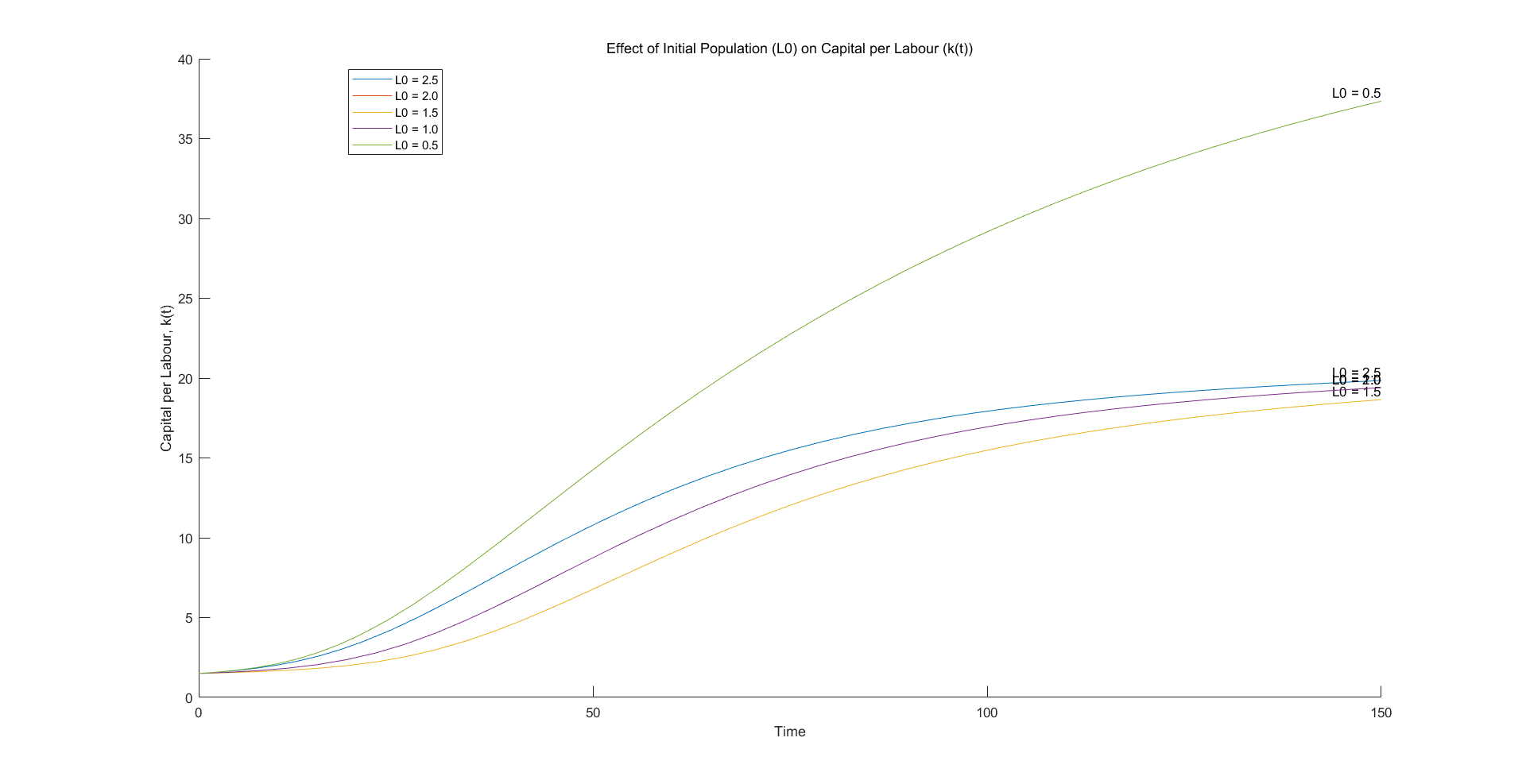}
    \caption{$k(t)$ with different initial values of $L_0$}
    \label{k-comparing}
\end{figure}
The capital per labour stabilizes at a higher level in case of $L_0 < N$ compared to $L_0 > N$.
Also the consumption per labour stabilizes at a higher level, see Figure \ref{c-comparing}.
\begin{figure}[htp]
    \centering
    \includegraphics[width=0.8\textwidth,height=0.6\textwidth]{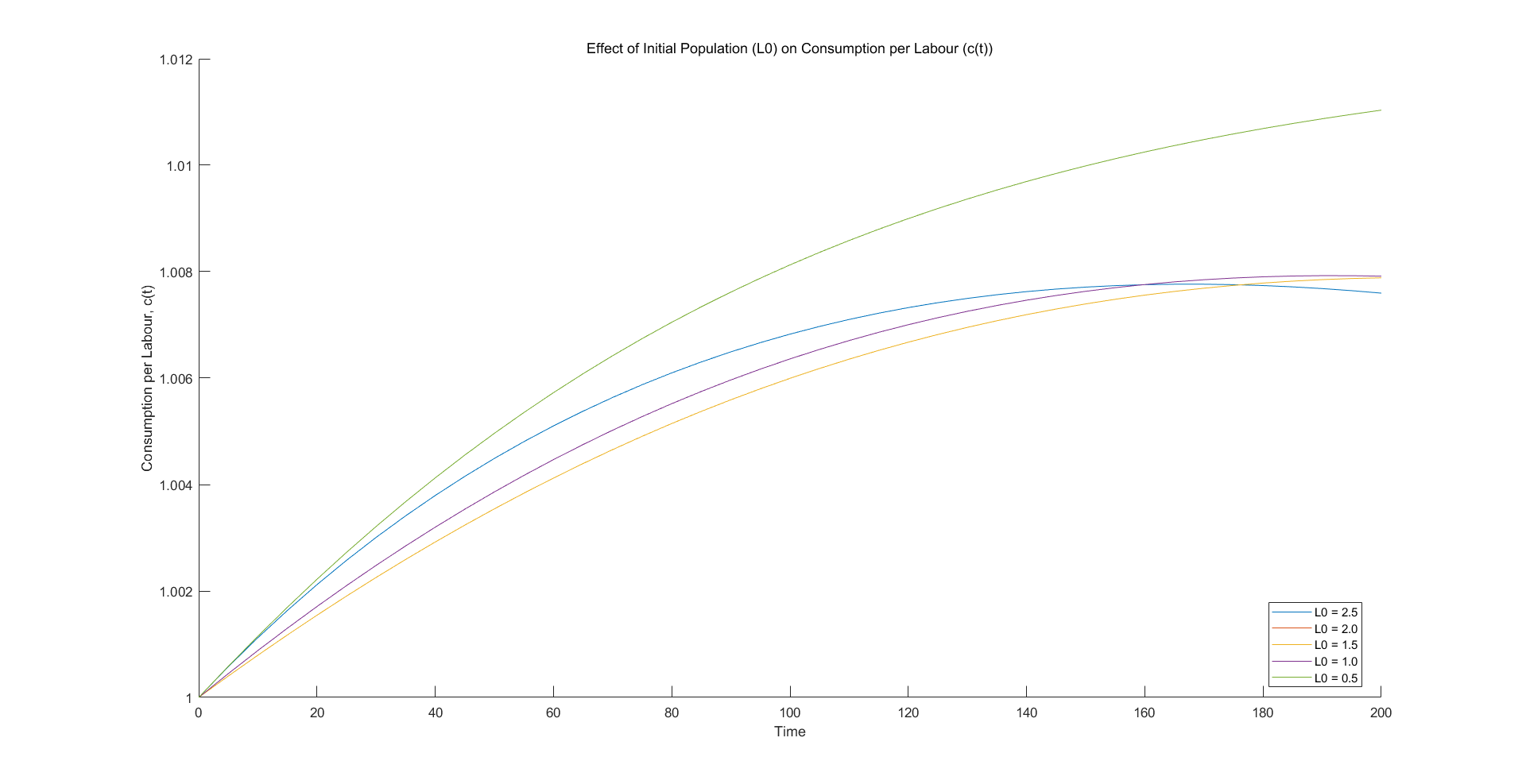}
    \caption{$c(t)$ with different initial values of $L_0$}
    \label{c-comparing}
\end{figure}

In order to illustrate the blow-up effect, which occurs if the stability condition $\delta>r$ is violated, let us take $r=0.085.$
If starting below the Allee threshold $N$, the capital per labour explodes, see Figure \ref{ex-blow}.

\begin{figure}[htp]
    \centering
    \includegraphics[width=0.8\textwidth,height=0.6\textwidth]{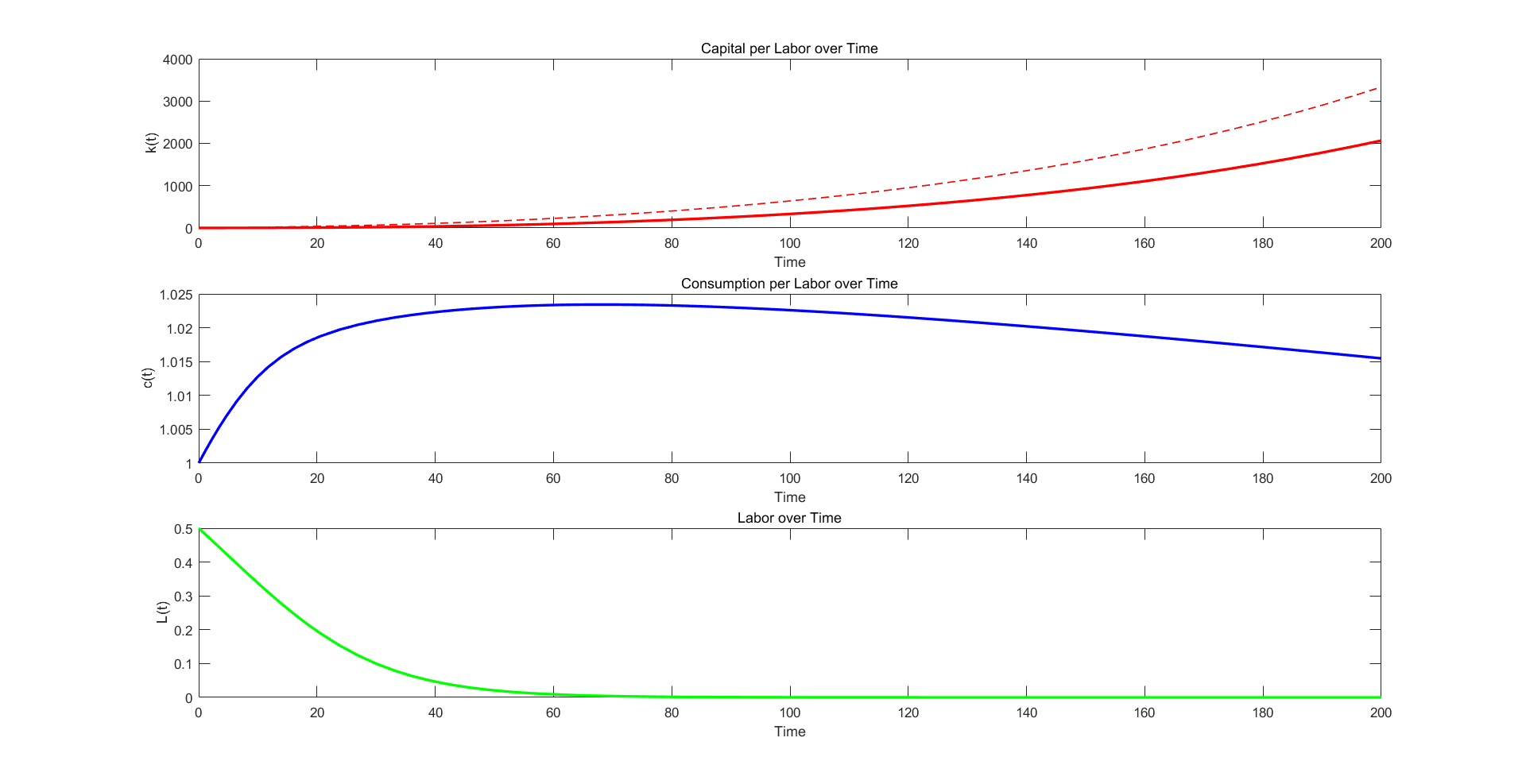}
    \caption{Blow-up effects}
    \label{ex-blow}
\end{figure}

\medskip
Received xxxx 20xx; revised xxxx 20xx; early access xxxx 20xx.
\medskip

\end{document}